\documentclass[a4paper]{article}
\usepackage[utf8]{inputenc}

\usepackage[normalem]{ulem}

\usepackage[margin=1in]{geometry}

\usepackage{amsmath, amssymb, mathtools}
\usepackage[hidelinks]{hyperref}

 \usepackage[dvipsnames]{xcolor}

\usepackage{amsthm}
\usepackage[noabbrev]{cleveref}
\newtheorem{thm}{Theorem}[section]
\newtheorem{lm}[thm]{Lemma}

\newtheorem{crl}[thm]{Corollary}
\newtheorem{prop}[thm]{Proposition}
\newtheorem{const}[thm]{Construction}

\theoremstyle{definition}

\newtheorem{rmk}[thm]{Remark}
\newtheorem{df}[thm]{Definition}

\newtheorem{ex}[thm]{Example}

\renewcommand{\phi}{\varphi}

\newcommand{\FF}{\mathbb F}

\newcommand{\QQ}{\mathbb Q}
\newcommand{\ZZ}{\mathbb Z}

\renewcommand{\leq}{\leqslant}
\renewcommand{\geq}{\geqslant}

\newcommand{\vspan}[1]{\left \langle #1 \right \rangle}
\newcommand{\vspann}[2]{\left \langle #1 \, \| \, #2 \right \rangle}
\newcommand{\set}[1]{ \left \{ #1 \right \} }
\newcommand{\sett}[2]{ \left\{ #1 \, \, || \, \, #2 \right \} }

\newcommand{\one}{\mathbf 1}
\newcommand{\zero}{\mathbf 0}
\newcommand{\floor}[1]{\left \lfloor #1 \right \rfloor}
\newcommand{\ceil}[1]{\left \lceil #1 \right \rceil}

\newcommand{\mb}{\mathcal B}

\newcommand{\ml}{\mathcal L}

\renewcommand{\mp}{\mathcal P}

\newcommand{\ms}{\mathcal S}

\DeclareMathOperator{\dir}{dir}

\definecolor{zscolor}{RGB}{255,128,0}

\definecolor{csbcolor}{RGB}{0,0,255}

 \DeclareMathOperator{\Tr}{Tr}
    
 \DeclareMathOperator{\PG}{PG}
    
 \DeclareMathOperator{\AG}{AG}
    \newcommand{\ag}{\AG}

 \DeclareMathOperator{\GL}{GL}

 \newcommand{\fq}{\FF_q}
 \newcommand{\fqi}{\FF_q \cup \set \infty}
 \DeclareMathOperator{\Lift}{Lift}
 \newcommand{\ns}{\overline S}
 \newcommand{\zq}{\ZZ / q \ZZ}

\makeatletter
\let\@fnsymbol\@arabic  
\makeatother

\title{On equidistributed directions in finite affine planes}
\author{%
Sam Adriaensen%
\thanks{Department of Mathematics and Data Science,
Vrije Universiteit Brussel, Pleinlaan 2,
1050 Elsene, Belgium.
Email: \texttt{sam.adriaensen@vub.be}.}
\and
Bence Csajb\'ok%
\thanks{Department of Computer Science,
ELTE E\"otv\"os Lor\'and University,
H-1117 Budapest, P\'azm\'any P.\ stny.\ 1/C, Hungary.
Emails: \texttt{bence.csajbok@ttk.elte.hu} (Bence Csajb\'ok);
\texttt{zsuzsa.weiner@gmail.com} (Zsuzsa Weiner).}
\and
Zsuzsa Weiner%
{%
\thanks{Prezi.com}}
}
\date{}

\begin{document}

\maketitle

\begin{abstract}
 We investigate a recently introduced generalisation of determined directions in affine planes.
 A direction $(d)$ in an affine plane is a point of the line at infinity in the projective completion, and hence corresponds to a parallel class.
 A (multi)set $S$ of points is equidistributed from direction $(d)$ if all lines from that parallel class intersect $S$ in the same number of points.
 In this paper, we give a construction of point (multi)sets that are inequidistributed from exactly 3 directions in any finite translation plane, and prove that all such (multi)sets arise from this construction, generalising a result of Kiss and Somlai in Desarguesian planes of prime order.
 
We also use ideas from algebraic graph theory to prove results
on directions associated with a pair of point sets $S$ and $T$.
If for every direction $(d)$ either $S$ or $T$ is equidistributed
from $(d)$, we prove that $|S \cap T| = |S||T|/q^2$,
where $q$ is the order of the affine plane. 
 In the appendix, we give a combinatorial alternative approach to proving this equality, which can be of independent interest.
 We also introduce the notion of directions cross-determined by a pair of sets $S$ and $T$, and prove that $|S| |T| \leq q^2$ if $(S,T)$ cross-determines at most half of the directions, where equality forces $S = T$.
\end{abstract}

\section{Introduction}

\subsection{Previous work}

In this paper, we study finite axiomatic affine planes.
We use the notation $A_q = (\mp,\ml)$ for an affine plane of order $q$ with point set $\mp$ and line set $\ml$, where each line is a set of $q$ points.
The classical Desarguesian plane arising from the finite field $\fq$ of order $q$ will be denoted by $\ag(2,q)$.
In that case, the point set is $\fq^2$.

An affine plane of order $q$ can be extended uniquely to a projective plane of order $q$.
This requires adding $q+1$ new points that are on a new line $\ell_\infty$.
The points of $\ell_\infty$ are called \emph{directions}, and correspond to the parallel classes of lines in the affine plane.
We say that a line $\ell$ of the affine plane $A_q$ has direction $(d) \in \ell_\infty$ if $\ell$ contains $(d)$ in the projective extension.
Note that in $\ag(2,q)$, a parallel class consists of all lines with the same slope, which explains the terminology direction.
In that case, we represent the directions as $(d)$ with $d \in \fqi$.

\begin{df}
 A set $S$ of points in an affine plane $A_q$ is said to \emph{determine} a direction $(d) \in \ell_\infty$ if there exist 2 distinct points $P,Q \in S$ such that the line $PQ$ has direction $(d)$.
\end{df}

Note that if a set of points $S$ contains more than $q$ points, with $q$ the order of the plane, then by the pigeonhole principle, $S$ determines every direction.
The central question is determining which sets of points whose size is equal or close to $q$ determine only a small number of directions.
This has been mostly studied in Desarguesian planes, starting with the work of Rédei \cite{Redei}, and culminating in the seminal papers of Blokhuis, Ball, Brouwer, Storme, and Sz\H onyi \cite{BBBSS} and Ball \cite{Ball}.

More recently, Ghidelli \cite{Ghidelli} introduced a sensible generalisation of the problem to sets of points whose size exceeds the order $q$ of the affine plane.
He calls a direction $(d)$ \emph{special} for a set $S$ of points in $A_q$ if there exists a line $\ell$ with direction $(d)$ that does not intersect $S$ in $\floor{|S|/q}$ or $\ceil{|S|/q}$ points.
Note that if $|S| \leq q$, then special directions coincide with determined directions.
Ghidelli proved the following extension of a theorem of Rédei \cite{Redei} and Sz\H onyi \cite{Szonyi} concerning Desarguesian planes of prime order.

\begin{thm}[{\cite[Theorem 1.3]{Ghidelli}}]
 Let $p$ be a prime number, and $S$ a set of $kp-r$ points in $\ag(2,p)$ with $0 \leq r < p$ and $k \geq 1$.
 Then either $S$ is contained in the union of $k$ lines, or there are at least $\frac{p+k+2-r}{k+1}$ special directions.
\end{thm}

Ghidelli \cite[Problem 1.4]{Ghidelli} asked whether the bound can be improved to $\frac{p+3-r}{2}$.
This was answered in the negative by Kiss and Somlai \cite{Kiss:Somlai}, which we will explain shortly.
We first introduce some terminology.

\bigskip 

Many of the results in this paper do not only hold for sets of points, but also for multisets of points in affine planes $A_q$.
An \emph{integer multiset $S$} of points in $A_q$ is defined by a multiplicity function $\mu_S: \mp \to \mathbb Z$.
We do allow negative multiplicities.
This describes an ordinary set if and only if $\mu_S$ only takes values in $\{0,1\}$.
We define the size of $S$ and its intersection size with a line $\ell$ by taking multiplicity into account:
\begin{align*}
 |S| = \sum_{P \in \mp} \mu_S(P), &&
 |S \cap \ell| = \sum_{P \in \ell} \mu_S(P).
\end{align*}

\begin{df}
 Suppose that $S$ is an integer multiset of points in an affine plane $A_q$ of order $q$.
 Let $(d)$ be a direction.
 Then $S$ is \emph{equidistributed} from direction $(d)$ if all lines with direction $(d)$ intersect $S$ in the same number of points (counted with multiplicity), which is necessarily $|S|/q$, and \emph{inequidistributed} otherwise.
 We will also say
 that direction $(d)$ is equidistributed, respectively inequidistributed,
 for $S$.
\end{df}

Note that $S$ can only be equidistributed from some direction if $q$ divides $|S|$.

\begin{rmk}
 In previous articles \cite{Ghidelli, Kiss:Somlai, ASW}, the term special was used instead of inequidistributed. For point sets whose size is divisible by the order of the plane, the two notions coincide. 
 However, being equidistributed from a direction $(d)$ actually stands out, and being inequidistributed is rather unremarkable.
 Therefore, calling such directions special feels counterintuitive, and we avoid this terminology.
\end{rmk}

We now present the remarkable result of Kiss and Somlai, who considered ordinary sets.

\begin{thm}[{\cite{Kiss:Somlai}}]
 \label{Thm:KS}
 Let $p$ be a prime number, and consider $\ag(2,p)$.
 Interpret the elements of the field $\FF_p$ as the integers $0,\dots,p-1$, naturally equipped with the order relation $<$.
 The set
 \[
  K = \sett{(x,y) \in \FF_p^2}{ x < y}
 \]
 is inequidistributed from only 3 directions: $(\infty)$, $(0)$, and $(1)$.
 Moreover, any set in $\ag(2,p)$ that is inequidistributed from exactly 3 directions is the image of $K$ or its complement under an affine transformation.
\end{thm}

Very recently, Ghidelli, Kiss, and Somlai \cite{Ghidelli:Kiss:Somlai} proved a strong non-existence result on sets with few inequidistributed directions in Desarguesian planes of prime order, answering a question posed in \cite{ASW}.

\begin{thm}[{\cite{Ghidelli:Kiss:Somlai}}]
 Suppose that $p$ is a prime number, and $k$ is an integer with 
 \[
  4 \leq k \leq \frac 18 \frac{\sqrt p}{2 + \log p}.
 \]
 Then $\AG(2,p)$ has no point sets with exactly $k$ inequidistributed directions.
\end{thm}

Not much is known about equidistributed directions in planes other than
$\ag(2,p)$.
Beyond the classical problem of determined directions, equidistributed
directions in $\ag(2,q)$ for non-prime $q$ were studied in \cite{ASW}.
A related notion of uniformity was investigated in \cite{CSW}, where
all but a small number of lines in a parallel class are required to
meet a multiset in the same number of points modulo the characteristic.
To the best of our knowledge, equidistributed directions have not
previously been studied in non-Desarguesian affine planes.

\subsection{Main results}

In this paper, we study the direction problem in not necessarily Desarguesian affine planes.
In Section \ref{Sec:characteristic}, we use some ideas from algebraic graph theory to prove the following theorem.

\begin{thm}
 \label{Thm:No shared directions}
 Let $A_q = (\mp,\ml)$ be an affine plane of order $q$.
 Let $S$ and $T$ be integer multisets of points in $A_q$.
 Suppose that there is no direction from which both $S$ and $T$ are inequidistributed.
 Then
 \[
  \sum_{P \in \mp} \mu_S(P) \mu_T(P) = \frac{|S| |T|}{q^2}.
 \]
 If $S$ and $T$ are ordinary sets, this formula simplifies to $|S \cap T| = |S| |T| / q^2$.
 In particular, if $S$ and $T$ are disjoint point sets of size $q$, then there is a direction determined by both of them.
\end{thm}
We give an alternative proof of Theorem \ref{Thm:No shared directions} in Appendix \ref{Sec:Covariance} that avoids linear algebra and extends the classical variance method for studying a single set to a covariance method for studying pairs of sets.

\bigskip 

In Section \ref{Sec:cross}, we extend the notion of determined directions to pairs of sets. This will give a bipartite analogue of the classical direction problem: instead of considering directions determined by pairs of points inside one set, we consider directions determined by pairs of points from two possibly different sets.

\begin{df}
 Suppose that $S$ and $T$ are sets of points in an affine plane $A_q$.
 We say that $(S,T)$ \emph{cross-determines} a direction $(d)$ if there are distinct points $P \in S$ and $Q \in T$ such that the line $PQ$ has direction $(d)$.
\end{df}

We prove the following theorem.

\begin{thm}
 \label{Thm:Cross determine}
 Suppose that $S$ and $T$ are sets of points in an affine plane $A_q$ of order $q$ and the number of directions cross-determined by $(S,T)$ equals $c \leq \frac{q+1}2$.
 Then $|S| |T| \leq q^2$.
 If $|S| |T| = q^2$, then $S = T$ is a set of $q$ points determining $c$ directions.
\end{thm}

We show that, in every Desarguesian affine plane, the hypothesis
$c\leq\frac{q+1}{2}$ is sharp for the conclusion that equality
forces $S=T$.

\bigskip

In Section \ref{Sec:3Directions}, we classify integer multisets with exactly three
inequidistributed directions in Desarguesian affine planes. 
After a suitable affine transformation, the multiplicity function
of every such multiset is an integer linear combination of the
multiplicity functions of lines in the three inequidistributed
directions, possibly plus the multiplicity function of a lift
of the Kiss--Somlai set.

For ordinary point sets in $\ag(2,q)$, this yields the following theorem.

\begin{thm}
 \label{Thm:Main:Lift}
 Let $p$ be a prime number and $q = p^h$ a power of $p$.
 Let $\Tr:\FF_q\to\FF_p$ be the absolute trace function.
 Recall the natural total ordering $<$ on the elements of $\FF_p$.
 Then the set 
 \[
  L = \sett{(x,y) \in \FF_q^2}{\Tr(x) < \Tr(y)}
 \]
 of points in $\ag(2,q)$ is inequidistributed from exactly 3 directions: $(0)$, $(1)$, and $(\infty)$.
 Moreover, every set that is inequidistributed from exactly 3 directions in $\ag(2,q)$ is affinely equivalent to the union of $L$ and some parallel lines that are disjoint from $L$.
\end{thm}

We explain how this construction and the classification result carry over to all translation planes.

\section{Equidistribution and characteristic functions}
 \label{Sec:characteristic}

The multiplicity function $\mu_S$ of an integer multiset $S$ is a vector in the vector space $\QQ^\mp$ of functions $\mp \to \QQ$.
In particular, for any line $\ell$, we have
\[
 \mu_\ell: \mp \to \QQ: P \mapsto \begin{cases}
  1 & \text{if } P \in \ell, \\
  0 & \text{otherwise}.
 \end{cases}
\]
We give an orthogonal decomposition of $\QQ^\mp$ according to the parallel classes of the affine plane.
This decomposition gives a useful way to detect equidistributed directions, leading to Theorem \ref{Thm:No shared directions}, and will later be used to transfer constructions between affine planes sharing some parallel classes.
The same ideas were used by Kiss and Somlai \cite[Proposition 5.1]{Kiss:Somlai}.

We denote the constant function taking 1 everywhere by $\one$ and the zero function by $\zero$.
We endow $\QQ^\mp$ with the standard scalar product given by
\[
 f \cdot g = \sum_{P \in \mp} f(P) g(P).
\]
If $\ell$ has direction $(d)$, we write $\dir(\ell) = (d)$.

\begin{lm}
 \label{Lm:Ortho}
 Let $A_q = (\mp,\ml)$ be an affine plane of order $q$, and $\ell_\infty$ be the line at infinity.
 For every $(d) \in \ell_\infty$ define $V_{(d)} = \vspann{\mu_\ell - \frac 1q \one}{\ell \in \ml,\ \dir(\ell) = (d)}$.
 \begin{enumerate}
  \item \label{Lm:Ortho:1} $\vspan \one$ and the spaces $V_{(d)}$ for $(d) \in \ell_\infty$ form an orthogonal decomposition of $\QQ^\mp$.
  \item \label{Lm:Ortho:2} Let $S$ be an integer multiset of points and $(d)$ a direction.
  Then $\mu_S \perp V_{(d)}$ if and only if $S$ is equidistributed from $(d)$.
  Thus, if $D$ is the set of inequidistributed directions of $S$, then $\mu_S \in \vspann{\one,V_{(d)}}{(d) \in D}$.
 \end{enumerate}
\end{lm}

\begin{proof}
 (1) First let us check that these subspaces are mutually orthogonal.
 For every line $\ell$,
\[
 \left(\mu_\ell - \frac 1q \one \right) \cdot \one = q - \frac{q^2}q = 0,
\]
hence $\vspan \one$ is orthogonal to all spaces $V_{(d)}$.
Now take two affine lines $\ell_1$ and $\ell_2$ with different directions.
Then they have one point in common, hence
\[
 \left(\mu_{\ell_1} - \frac 1q \one \right) \cdot \left(\mu_{\ell_2} - \frac 1q \one \right) = \mu_{\ell_1} \cdot \left(\mu_{\ell_2} - \frac 1q \one \right) - \frac 1 q \one \cdot \left(\mu_{\ell_2} - \frac 1q \one \right) = \left(1 - \frac 1 q q \right) - 0 = 0.
\]
This shows that the spaces $V_{(d)}$ are mutually orthogonal.

It remains to show that $\vspann{\one, V_{(d)}}{(d) \in \ell_\infty} = \QQ^\mp$. 
First note that $\vspann{\one, V_{(d)}}{(d) \in \ell_\infty} = \vspann{\mu_\ell}{\ell \in \ml}$ is the column space of the point-line incidence matrix $B$ of $A_q$.
It is well known that $B B^\top = q I + J$ with $J$ the all-one matrix, which implies that $B$ has rank $q^2$, thus its column space is $\QQ^\mp$.

\bigskip

(2) The integer multiset $S$ is equidistributed from direction $(d)$ if and only if for all affine lines $\ell \ni (d)$ we have
\[
 \left( \mu_\ell - \frac 1q \one \right) \cdot \mu_S = \frac{|S|}q - \frac 1 q |S| = 0,
\]
which is equivalent to $\mu_S \in V_{(d)}^\perp$.
The statement now follows from (1).
\end{proof}


We now apply the decomposition from Lemma~\ref{Lm:Ortho} to pairs of point sets whose inequidistributed directions are disjoint.

\bigskip

\noindent {\bf Theorem \ref{Thm:No shared directions}.}
{\it  Let $A_q = (\mp,\ml)$ be an affine plane of order $q$.
 Let $S$ and $T$ be integer multisets of points in $A_q$.
 Suppose that there is no direction from which both $S$ and $T$ are inequidistributed.
 Then
 \[
  \sum_{P \in \mp} \mu_S(P) \mu_T(P) = \frac{|S| |T|}{q^2}.
 \]
 If $S$ and $T$ are ordinary sets, this formula simplifies to $|S \cap T| = |S| |T| / q^2$.
 In particular, if $S$ and $T$ are disjoint point sets of size $q$, then there is a direction determined by both of them. }

\begin{proof}
 Let $\mu_S^*$ and $\mu_T^*$ denote the orthogonal projections of $\mu_S$ and $\mu_T$ onto $\vspan \one ^\perp$.
 Then
 \[
  \mu_S = \frac{\mu_S \cdot \one}{\one \cdot \one} \one + \mu_S^* = \frac{|S|}{q^2} \one + \mu_S^*.
 \]
 Similarly, $\mu_T = \frac{|T|}{q^2} \one + \mu_T^*$.
 By Lemma \ref{Lm:Ortho}, $\mu_S^*$ and $\mu_T^*$ are orthogonal.
 It follows that
 \[
  \sum_{P \in \mp} \mu_S(P) \mu_T(P) = \mu_S \cdot \mu_T = \left(\frac{|S|}{q^2} \one + \mu_S^*\right) \cdot \left(\frac{|T|}{q^2} \one + \mu_T^*\right ) = \frac{|S| |T|}{q^2}.
 \]
 The remainder of the theorem statement is an immediate consequence.
\end{proof}

Lemma \ref{Lm:Ortho} also allows us to compare equidistributed directions in affine planes which share some parallel classes. This will be used in Section \ref{Sec:3Directions}.

\begin{prop}
 \label{Prop:Shared directions}
 Suppose that $A_q = (\mp,\ml)$ and $A'_q = (\mp,\ml')$ are affine planes of order $q$ on the same point set.
 Suppose that $\ml$ and $\ml'$ share some parallel classes of lines.
 We may assign the same direction to each shared parallel class of lines in $A_q$ and $A_q'$; let $D$ denote the set of shared directions.
 Let $S$ be an integer multiset of points.
 Suppose that in $A_q$, the inequidistributed directions of $S$ are a subset $D' \subseteq D$.
 Then also in $A'_q$, $D'$ is the set of inequidistributed directions of $S$.
\end{prop}

\begin{proof}
 This follows directly from Lemma \ref{Lm:Ortho}.
\end{proof}

\begin{rmk}
 The results in this section can also be proven with elementary counting arguments.
 We refer the reader to Appendix \ref{Sec:Covariance}.
 There, we extend the classical variance method from design theory to a covariance method.
\end{rmk}

\section{Cross-determined directions}
\label{Sec:cross}

In this section, we extend the notion of determined directions to pairs of point
sets and prove Theorem~\ref{Thm:Cross determine}. This gives a bipartite analogue of the classical direction problem.

\subsection{The Hoffman ratio bound for cross-cocliques}

The proof of Theorem~\ref{Thm:Cross determine} uses a Hoffman ratio-type bound for cross-cocliques.
A \emph{cross-coclique} in a graph $G$ is a pair $(S,T)$ of sets of vertices such that no vertex in $S$ is adjacent to a vertex in $T$.
We allow $S \cap T$ to be non-empty, in which case it must be a coclique. 
The bound we use is known.
The proof can be found e.g.\ in \cite[Theorem 13]{Ellis_et_al}, but can also be obtained by applying the bipartite expander mixing lemma to the bipartite double of $G$, together with the AM-GM inequality.
However, we need a stronger characterisation result in the case of equality, so we include a full proof here.

\begin{prop}
 \label{Prop:General CS}
 Let $B$ be a symmetric real matrix with distinct eigenvalues $\mu_1 > \dots > \mu_d$.
 Then for any two vectors $x$ and $y$, 
  \[
   |x^\top B y| \leq \max\{|\mu_1|, |\mu_d|\} \| x \| \| y \|.
  \]
 If $\mu_d< - \mu_1$, then equality holds if and only if $x$ and $y$ are linearly dependent $\mu_d$-eigenvectors of $B$ or one of the vectors is $\zero$.
\end{prop}

\begin{proof}
 We will suppose that $x$ and $y$ are non-zero, otherwise the claim is trivially true.
 Since $B$ is symmetric, its eigenspaces form an orthogonal decomposition.
 Suppose that $x = \sum_{i=1}^d x_i$ and $y = \sum_{i=1}^d y_i$ with $x_i$ and $y_i$ in the $\mu_i$-eigenspace of $B$ for each $i$.
 Then
 \[
  |x^\top B y| = \left| \sum_{i=1}^d \mu_i x_i^\top y_i \right| \leq \sum_{i=1}^d |\mu_i| | x_i^\top y_i | \leq \left( \max_{i=1}^d |\mu_i| \right) \sum_{i=1}^d \| x_i \| \|y_i \|,
 \]
 where the last inequality holds by the Cauchy-Schwarz inequality.
 Note that we have $\max_{i=1}^d |\mu_i| = \max\{|\mu_1|, |\mu_d|\}$.
 Since the vectors $x_i$ are pairwise orthogonal, we have that $\sum_{i=1}^d \|x_i\|^2 = \| x \|^2$ and likewise $\sum_{i=1}^d \|y_i\|^2 = \|y\|^2$.
 By applying the Cauchy-Schwarz inequality to the vectors $(\|x_1\|, \dots, \|x_d\|)$ and $(\|y_1\|, \dots, \|y_d\|)$, we see that $\sum_{i=1}^d \|x_i\| \|y_i\| \leq \|x\| \|y\|$.
 This proves the inequality from the statement.

 Now suppose that $\mu_d < - \mu_1$, which means that $\max\{|\mu_1|, |\mu_d|\} = - \mu_d$, and also assume that equality holds.
 Then we have $\sum_{i=1}^d |\mu_i| | x_i^\top y_i | = -\mu_d \sum_{i=1}^d \| x_i \| \|y_i \|$.
 This means that $\|x_i \| \| y_i\| = 0$ whenever $i < d$.
 Moreover, the Cauchy-Schwarz inequality holds for the vectors $(\|x_1\|, \dots, \|x_d\|)$ and $(\|y_1\|, \dots, \|y_d\|)$ with equality, which means that $\|x_i\| = \|y_i\| = 0$ whenever $i < d$.
 It follows that $x$ and $y$ are $\mu_d$-eigenvectors of $B$, i.e.\ $x = x_d$ and $y = y_d$.
 Lastly, we must have that $| x_d^\top y_d | = \| x_d \| \| y_d \|$, which implies that $x_d$ and $y_d$ are linearly dependent.
Conversely, if $x$ and $y$ are linearly dependent $\mu_d$-eigenvectors, then clearly $|x^\top B y| = |\mu_d| \|x\| \|y\|$.
\end{proof}

\begin{thm}
 \label{Thm:Cross coclique}
 Let $G$ be a $k$-regular graph on $n$ vertices, $k>0$, with eigenvalues of the adjacency matrix given by $\lambda_1 \geq \dots \geq \lambda_n$ (counted with multiplicity).
 Let $\lambda = \max\{|\lambda_2|, |\lambda_n|\}$.
 If $(S,T)$ is a cross-coclique in $G$ then
 \[
  |S| |T| \leq \left( \frac{n}{\frac{k}{\lambda } + 1} \right)^2.
 \]
 If $\lambda_n < -\lambda_2$ and equality holds in the bound, then $S = T$.
\end{thm}

\begin{proof}
 Let $\mu_S$ and $\mu_T$ be the multiplicity functions of the sets $S$ and $T$ respectively, seen as subsets of the vertices of $G$.
 Interpret $\mu_S$ and $\mu_T$ as column vectors, with coordinates labelled by the vertices of $G$.
 We will assume that $S$ and $T$ are non-empty to avoid trivialities.
 Let $A$ be the adjacency matrix of $G$.
 Then the fact that $(S,T)$ is a cross-coclique means that $\mu_S^\top A \mu_T = 0$.
 When we project everything onto the orthogonal complement of $\one$, this becomes
 \begin{equation}
  \label{Eq:Projection}
  \left( \mu_S - \frac{|S|}n \one \right)^\top \left( A - \frac kn J \right) \left( \mu_T - \frac{|T|}n \one \right) = - \frac k n |S| |T|.
 \end{equation}
 Note that the only difference in the spectra of $A$ and $A - \frac kn J$ is that we change the eigenvalue of $\one$ from $k$ to 0, hence $\lambda_1$ becomes 0, and the other eigenvalues stay the same.
 By applying Proposition \ref{Prop:General CS} to \eqref{Eq:Projection}, we obtain
 \begin{equation}
  \label{Eq:Cross proof}
  \frac kn |S| |T| \leq \lambda \left \| \mu_S - \frac{|S|}n \one \right \| \left \| \mu_T - \frac{|T|}n \one \right \| = \lambda \sqrt{ |S| \frac{n-|S|}n} \sqrt{|T| \frac{n-|T|}n}.
 \end{equation}
 Now we use the following inequality, which one can see as the AM-GM inequality.
 \begin{align*}
  \left(n-\sqrt{|S| |T|}\right)^2 - (n-|S|)(n-|T|) &= \left(n^2 - 2 \sqrt{|S| |T|} n + |S| |T|\right) - (n^2 - (|S|+|T|)n + |S| |T|) \\
  &= \left(|S| - 2 \sqrt{|S| |T|} + |T|\right)n = \left(\sqrt{|S|} - \sqrt{|T|}\right)^2 n \geq 0.
 \end{align*}
 We conclude that $(n-|S|)(n-|T|) \leq (n-\sqrt{|S| |T|})^2$ with equality if and only if $|S| = |T|$.
 Plugging this into \eqref{Eq:Cross proof}, we find that
 \[
  \frac kn |S| |T| \leq \frac \lambda n \sqrt{|S| |T|}(n-\sqrt{|S| |T|}).
 \]
 This inequality is equivalent to
 \[
  \sqrt{|S| |T|} \leq \frac n {\frac k \lambda + 1 },
 \]
 proving the inequality of the theorem.

 Now suppose that $\lambda_n < - \lambda_2$.
 If equality holds, then we must have $|S| = |T|$ and in (\ref{Eq:Cross proof}) we find equality for Proposition \ref{Prop:General CS}, hence $\mu_S - \frac{|S|}n \one$ and $\mu_T - \frac{|T|}n \one$ must be linearly dependent $\lambda_n$-eigenvectors of $A$.
 Since $|S| = |T|$, these vectors have the same magnitude, which means that they are either the same vector, or differ by a factor $-1$.
 Suppose that the latter holds, and write $\mu_S - \frac{|S|}n \one = x$.
 Then we have $\mu_S = \frac{|S|}n \one + x$ and $\mu_T = \frac{|S|}n \one - x$.
 Note that both of these vectors are $\{0,1\}$-vectors.
 If there exists a vertex $v$ with $x(v) = 0$, then $|S|/n$ is either $0$ or $1$.
The former contradicts the assumption that $S$ is non-empty. The latter implies that $S$ is the full vertex set.
 Since $T$ is non-empty, $G$ must have isolated vertices.
 But given that $G$ is $k$-regular, this means that $k=0$, hence $\lambda_2 = \lambda_n = 0$, contradicting that $\lambda_n < - \lambda_2$.
 This means that $x(v)$ is non-zero everywhere.
 Then we have $\{|S|/n \pm x(v)\} = \{0,1\}$, hence $|S| = n/2$ and $x(v) = \pm 1/2$.
 This immediately implies that $S$ and $T$ partition the vertex set.
 But since they form a cross-coclique, this implies that $G$ is disconnected.
 The multiplicity of $k$ as eigenvalue of $A$ equals the number of connected components.
 Hence we find that $\lambda_2 = k$.
 But then $\lambda_n < -\lambda_2 = -k$ cannot hold since a $k$-regular graph cannot have an eigenvalue whose absolute value exceeds $k$ by the Perron-Frobenius theorem.

 The other option is that $\mu_S - \frac{|S|}n \one = \mu_T - \frac{|T|}n \one$, which together with $|S| = |T|$ implies that $S = T$.
\end{proof}

\subsection{A bound for cross-determined directions}

We now apply the cross-coclique bound to the graph defined by the directions not
cross-determined by a pair of point sets.

\bigskip 

\noindent {\bf Theorem \ref{Thm:Cross determine}.}
{\it  Suppose that $S$ and $T$ are sets of points in an affine plane $A_q$ of order $q$ and the number of directions cross-determined by $(S,T)$ equals $c \leq \frac{q+1}2$.
 Then $|S| |T| \leq q^2$.
 If $|S| |T| = q^2$, then $S = T$ is a set of $q$ points determining $c$ directions.
}

\begin{rmk}
 By applying this theorem in Desarguesian planes, if equality holds, $S$ must be a translate of a set of $q$ points which is a subspace over a subfield of $\FF_q$ by \cite{BBBSS, Ball}.
\end{rmk}

\begin{proof}[Proof of \Cref{Thm:Cross determine}]
 Assume that $|S| |T| > 1$, otherwise the theorem trivially holds.
 Then $(S,T)$ cross-determines at least one direction, that is $c > 0$.
 Let $D$ be the set of directions which are not cross-determined by $(S,T)$.
 Its size equals $b = q+1-c$.
 Make a graph $G$ defined on the points of $A_q$ where two distinct points $P$ and $Q$ are adjacent vertices if and only if the line $PQ$ has a direction in $D$.
 Then $G$ is a strongly regular graph with parameters $(q^2, b(q-1), b^2 - 3b + q, b(b-1))$.
 This graph is of Latin square type (and in fact arises from a set of $b-2$ mutually orthogonal Latin squares).
 The eigenvalues of $G$ are $b(q-1)$ with multiplicity 1, $q-b$ and $-b$.
 We refer the reader to \cite[\S 8.4]{Brouwer:VanMaldeghem} for more details.
 Given that $c \leq \frac{q+1}2$, we have $b \geq \frac{q+1}2$ and hence $-b < -(q-b)$.

 By construction of $G$, $(S,T)$ is a cross-coclique in $G$.
 We apply \Cref{Thm:Cross coclique} to find that
 \[
  |S| |T| \leq \left( \frac{q^2}{\frac{b(q-1)}b + 1} \right)^2 = q^2.
 \]
 In case of equality, we must have $S = T$.
\end{proof}


The following examples show that the bound on the number of
cross-determined directions is sharp for the conclusion
that equality implies $S=T$.

\begin{ex}
 Suppose that $q$ is a prime power and $A_q = \ag(2,q)$ is a Desarguesian affine plane.
\begin{enumerate}
 \item
 First suppose that $q$ is odd.
 We will build an example based on the projective triangle.
 Let $S_q$ denote the set of non-zero squares of $\FF_q$ and $\ns_q$ the set of non-squares.
 Define
 \begin{align*}
  S &= \{(0,0)\} \cup \sett{(x,0)}{x \in S_q} \cup \sett{(0,y)}{y \in S_q} \\
  T &= \{(0,0)\} \cup \sett{(x,0)}{x \in \ns_q} \cup \sett{(0,y)}{y \in \ns_q}
 \end{align*}
 Then $S$ and $T$ have size $q$ and $|S \cap T| = 1$.
 We will check that $(S,T)$  cross-determines $\frac{q+3}2$ directions.
 Indeed, the axes $X=0$ and $Y=0$ clearly both intersect $S$ and $T$ (in distinct points), hence $(0)$ and $(\infty)$ are cross-determined.
 If we take a point $P \in S$ and a different point $Q \in T$ such that the line $\ell$ spanned by $P$ and $Q$ has a different slope than $(0)$ or $(\infty)$, then the points we chose are $(x,0)$ and $(0,y)$ with one element of $\{x,y\}$ in $S_q$ and the other one in $\ns_q$.
 Then $\ell$ has slope $(-y/x)$, and $y/x \in \ns_q$.
 Every element of $-\ns_q$ occurs as such a slope, so the directions cross-determined by $(S,T)$ are $(0)$, $(\infty)$, and all $(-d)$ with $d \in \ns_q$, which indeed gives us $\frac{q+3}2$ directions.
 \item Now suppose that $q$ is even. We will build an example based on the projective triad.
 Consider the absolute trace function $\Tr: \FF_q \to \FF_2$ on $\FF_q$.
 Define
 \begin{align*}
 S &= \{0,1\} \times \sett{y \in \FF_q}{\Tr(y) = 0} \\
 T &= \{0,1\} \times \sett{y \in \FF_q}{\Tr(y) = 1}. 
\end{align*}
Then $S$ and $T$ both have size $q$ and are disjoint.
Clearly the lines $X=1$ and $X=0$ intersect both $S$ and $T$, hence $S$ and $T$ cross-determine $(\infty)$.
Now suppose that $P \in S$ and $Q \in T$ are joined by a non-vertical line $\ell$.
Then they have coordinates of the form $(0,y)$ and $(1,z)$ with $\Tr(y) \neq \Tr(z)$.
Hence, $\ell$ has slope $z-y$ which has trace 1.
Conversely, it is easy to check that any direction $(d)$ with $\Tr(d) = 1$ is cross-determined.
Hence, $(S,T)$ cross-determines $\frac q2 + 1$ directions.
\end{enumerate}
\end{ex}

\section{Multisets inequidistributed from three directions in translation planes}
\label{Sec:3Directions}

In this section, we construct and classify integer multisets of points in translation planes with exactly 3 inequidistributed directions.
We first treat the Desarguesian case using a lifting
construction and then transfer the results to arbitrary translation planes
using the Andr\'e/Bruck--Bose representation.

\subsection{A lifting construction}

We first consider the Desarguesian affine plane $\AG(2,q)$ with the directions $\fq\cup\{\infty\}$.

\begin{df}\label{df:multiset-lift}
Let $S$ be an integer multiset of points in $\AG(2,q)$, let $\FF_{q^e}$ be a field
extension of $\fq$, and consider the trace map
$\Tr\colon\FF_{q^e}\to\fq $.
$\Lift(S)$ in $\AG(2,q^e)$ is the integer multiset defined by
\[
    \mu_{\Lift(S)}(X,Y)
    =\mu_S\bigl(\Tr(X),\Tr(Y)\bigr).
\]
\end{df}

We will also interpret $\Tr$ as a map on vectors in $\FF_{q^e}^2$ by applying the trace map componentwise.

\begin{lm}\label{lm:lift-equivariance}
Suppose that $M\in\GL(2,q)$. Then, for every integer multiset $S$ in
$\AG(2,q)$,
\[
    \Lift(S^M)=\Lift(S)^M.
\]
More generally, lifts of affinely equivalent integer multisets
are affinely equivalent. 
\end{lm}

\begin{proof}
For $v\in\FF_{q^e}^2$, the multiplicity of $v$ in $\Lift(S)^M$ is
\[
    \mu_{\Lift(S)}(M^{-1}v)
    =\mu_S\bigl(\Tr(M^{-1}v)\bigr).
\]
Since the entries of $M^{-1}$ belong to $\fq$, the trace map commutes with
$M^{-1}$. Hence this equals
\[
    \mu_S\bigl(M^{-1}\Tr(v)\bigr)
    =\mu_{S^M}\bigl(\Tr(v)\bigr),
\]
which is the multiplicity of $v$ in $\Lift(S^M)$.

For translations, let $b\in\fq^2$ and choose
$\widetilde b\in\FF_{q^e}^2$ with $\Tr(\widetilde b)=b$,
which is possible by the surjectivity of the trace map.
Then $\Tr(v-\widetilde b)=\Tr(v)-b$, so
\[
    \Lift(S+b)=\Lift(S)+\widetilde b.
\]
Together with the linear case, this proves the affine statement.
\end{proof}

\begin{prop}\label{prop:lift-directions}
Suppose that $S$ is an integer multiset in $\AG(2,q)$. Then $\Lift(S)$ in
$\AG(2,q^e)$ has the same inequidistributed directions as $S$. In particular, $\Lift(S)$ is equidistributed from every direction in
$\FF_{q^e}\setminus\fq$.
\end{prop}

\begin{proof}
First let $(d)$ be a direction in $\fq\cup\{\infty\}$. If $\ell$ has
direction $(\infty)$ and equation $X=a$, and $\ell'$ is the line
$X=\Tr(a)$ in $\AG(2,q)$, then
\[
    |\Lift(S)\cap\ell|=q^{e-1}|S\cap\ell'|.
\]
The assertion follows for $(\infty)$ and hence, by Lemma 
\ref{lm:lift-equivariance}, for every $(d)$ with $d\in\fq$.

Now let  $\ell$ be a line  with slope $d \in \FF_{q^e}\setminus\fq$, and hence equation
$Y=dX+b$ for some $b\in\FF_{q^e}$. Then $\ell$ is an
$e$-dimensional affine $\fq$-subspace.
Define for $(\alpha,\beta) \in \FF_q^2$, $B_{\alpha,\beta} = \Lift(\{(\alpha,\beta)\})$.
The equations
$\Tr(X)=\alpha$ and $\Tr(Y)=\beta$ define affine hyperplanes when we
view $\FF_{q^e}^2$ as an $\fq$-vector space. Neither of these
hyperplanes is parallel to $\ell$, since $\ell$ has points with every
$X$-coordinate and every $Y$-coordinate. Therefore, either
$\ell\cap B_{\alpha,\beta}$ is an $(e-2)$-dimensional affine
$\fq$-subspace, or the restrictions of the equations
\[
    \Tr(X)=\alpha
    \qquad\text{and}\qquad
    \Tr(Y)=\Tr(dX+b)=\beta
\]
to $\ell$ define parallel, possibly equal, hyperplanes. The latter
happens if and only if $\Tr(X)$ and $\Tr(dX)$ are $\fq$-linearly dependent
functionals. 
This happens
if and only if $d\in\fq$. We conclude that
\[
    |\ell\cap B_{\alpha,\beta}|=q^{e-2}.
\]

By definition, $\mu_{\Lift(S)}$ is constant on each block
$B_{\alpha,\beta}$, with value $\mu_S(\alpha,\beta)$. Hence
\[
\begin{aligned}
    |\ell\cap\Lift(S)|
    &=
    \sum_{\alpha,\beta\in\fq}
    \mu_S(\alpha,\beta)|\ell\cap B_{\alpha,\beta}| \\
    &=
    q^{e-2}\sum_{\alpha,\beta\in\fq}\mu_S(\alpha,\beta)
    =q^{e-2}|S|.
\end{aligned}
\]
Thus $\Lift(S)$ is equidistributed from $(d)$.
\end{proof}

Let $D$ be a set of directions.
For a line $\ell$ of $\AG(2,q)$, let $\dir(\ell)$ denote its slope.
Identifying each affine line $\ell$ with its
multiplicity function $\mu_\ell$, denote by
\[
    \ZZ\ml(D)
    =\sett{
        \sum_{\substack{\ell\in\ml\\ \dir(\ell)\in D}}
        c_\ell\mu_\ell}
        {c_\ell\in\ZZ
      }
\]
the integral span of the affine lines whose directions belong to $D$.

\begin{rmk}\label{rmk:shifting}
Suppose that $D$ is non-empty and let $(d)\in D$. Since the lines of
direction $(d)$ partition the affine plane,
\[
    \one=\sum_{\substack{\ell\in\ml\\\dir(\ell)=(d)}}\mu_\ell
    \in\ZZ\ml(D).
\]
Consequently, for every integer multiset $S$ and every $c\in\ZZ$,
\[
\mu_S+c\mathbf{1}+\mathbb Z\mathcal L(D)
=\mu_S+\mathbb Z\mathcal L(D).
\]
In particular, shifting an integer multiset to a non-negative
multiset does not change its class modulo $\ZZ\ml(D)$.
\end{rmk}

Write $q=p^h$, where $p$ is prime, and let
$\Tr\colon\fq\to\FF_p$ denote the absolute trace map.
Let
\[
    \nu\colon\FF_p\to\{0,\ldots,p-1\}
\]
denote the natural choice of integer representatives, and let
\[
    K=\sett{(x,y)\in\FF_p^2 }{ \nu(x) < \nu(y)}
\]
be the Kiss--Somlai set from Theorem \ref{Thm:KS}.
The next lemma is a useful identity for the multiplicity function of $\Lift(K)$.

\begin{lm}\label{lm:KF-identity}
For $x,y\in\fq$,
\[
    \mu_{\Lift(K)}(x,y)
    =\frac{\nu(\Tr(y))-\nu(\Tr(x))+\nu(\Tr(x-y))}{p}.
\]
\end{lm}

\begin{proof}
Put $u=\nu(\Tr(x))$, $v=\nu(\Tr(y))$, and $w=\nu(\Tr(x-y))$.
Since $w\equiv u-v\pmod p$, we have $w=u-v$ if $u \geq v$, and $w=u-v+p$ if $u < v$. Thus the numerator on the right-hand side is $0$ if $(x,y) \notin \Lift(K)$ and $p$ if $(x,y) \in \Lift(K)$.
\end{proof}

We will now prove that every multiset with exactly 3 inequidistributed directions can be transformed, by adding and subtracting lines with inequidistributed directions, to either the empty set or a copy of a lifted Kiss-Somlai set.

\begin{thm}\label{thm:integer-three-directions}
Let $p$ be prime and $q=p^h$. Suppose that $S$ is an integer multiset
in $\AG(2,q)$ that is inequidistributed from exactly three directions.
Then there exists an affine transformation $\phi$ of $\AG(2,q)$ such that either $\mu_{S^\varphi} \in \ZZ \ml(\set{(\infty),(0),(1)})$ or
\[
 \mu_{S^\varphi} \in \mu_{\Lift(K)} + \ZZ \ml(\set{(\infty),(0),(1)}).
\]
\end{thm}

\begin{proof}
We may apply a first affine transformation so that $S$ is
inequidistributed from $(\infty)$, $(0)$, and $(1)$.
For $t\in\fq$, let $\ell_\infty(t)$, $\ell_0(t)$, and $\ell_1(t)$ denote
the lines $X=t$, $Y=t$, and $X-Y=t$, respectively, and let $f_d^*(t)$ denote the intersection number of $S$ with $\ell_d(t)$, i.e. 
\[
    f_d^*(t)=\sum_{P\in\ell_d(t)}\mu_S(P),
    \qquad d\in\{\infty,0,1\}.
\]

All other directions are equidistributed and hence have intersection number $|S|/q$. Counting, with
multiplicity, the points of $S$ on the lines through $(x,y)$ gives
\[
    |S| + q\mu_S(x,y)
    =f_\infty^*(x)+f_0^*(y)+f_1^*(x-y) + (q-2) \frac{|S|}{q}.
\]
Let $\lambda: =\frac{2|S|}{q}$.
Then
\begin{equation}\label{eq:qmu-relation}
    q\mu_S(x,y)
    =f_\infty^*(x)+f_0^*(y)+f_1^*(x-y)  
    -\lambda.
\end{equation}
Here $\lambda$ is an integer: if $q>2$, this follows from the existence of an equidistributed direction, while for $q=2$ we have $\lambda=|S|$.

For $a\in\ZZ$, write $[a]_q$ for its residue class in $\ZZ/q\ZZ$.
Reducing \eqref{eq:qmu-relation} modulo $q$ and subtracting its special
case $x=y=0$ gives
\[
    [f_\infty^*(x)-f_\infty^*(0)]_q
    +[f_0^*(y)-f_0^*(0)]_q
    +[f_1^*(x-y)-f_1^*(0)]_q=[0]_q.
\]
Define $g_d(t)=[f_d^*(t)-f_d^*(0)]_q\in\ZZ/q\ZZ$. Then
\begin{equation}\label{eq:g-relation}
    g_\infty(x)+g_0(y)+g_1(x-y)=[0]_q.
\end{equation}
Setting $y=0$ in \eqref{eq:g-relation} gives
$g_\infty(x)=-g_1(x)$, while setting $x=0$ gives
$g_0(y)=-g_1(-y)$. Substituting these identities back into
\eqref{eq:g-relation}, we obtain
\[
    g_1(x-y)=g_1(x)+g_1(-y)
\] 
for all $x,y\in\fq$, which means that $g_1$ is a group homomorphism from $(\FF_q,+)$ to $(\zq,+)$. The additive group of $\fq$ has exponent $p$, so $pg_1(t)=[0]_q$ for
every $t\in\fq$. Therefore, every value of $g_1$ is of the form
$[\frac qp i]_q$ for a unique $i\in\{0,\ldots,p-1\}$. This defines an
additive, and hence $\FF_p$-linear, map $F\colon\fq\to\FF_p$ satisfying
\[
    g_1(t)=\left[\frac qp\,\nu(F(t))\right]_q.
\]
Since $g_1$ is additive, we also have $g_\infty=-g_1$ and $g_0=g_1$.
Moreover, since $F$ is $\FF_p$-linear, it must be of the form $F(t) = \Tr(at)$ for some $a \in \fq$.
It follows that there are integer-valued functions
$A_\infty, A_0, A_1 \colon\fq\to\ZZ$, vanishing at zero, such that
\begin{align}
  \label{Eq:f's and A's}
 \begin{aligned}
    f_\infty^*(x)&=f_\infty^*(0)-\frac qp\nu(\Tr(ax))+q A_\infty(x),\\
    f_0^*(y)&=f_0^*(0)+\frac qp\nu(\Tr(ay))+q A_0(y),\\
    f_1^*(z)&=f_1^*(0)+\frac qp\nu(\Tr(az))+q A_1(z).
\end{aligned}
\end{align}
Substituting these expressions into \eqref{eq:qmu-relation} gives
\begin{equation}
\label{eq:muT}
    \mu_S(x,y)
    =\underbrace{\mu_S(0,0)+A_\infty(x)+A_0(y)+A_1(x-y)}_{\eqqcolon L(x,y)}
     +\frac{\nu(\Tr(ay))-\nu(\Tr(ax))+\nu(\Tr(a(x-y)))}{p}.
\end{equation}
Note that $L \in \ZZ \ml(\{(\infty), (0), (1) \})$.
If $a = 0$, then $\mu_S = L$.
Otherwise, let $\phi: \fq^2 \to \fq^2: (x,y) \mapsto a(x,y)$.
Then
\[
\mu_{S^\phi}=L\circ\phi^{-1}+\mu_{\Lift(K)}
\]
by Lemma~\ref{lm:KF-identity}.
Since $L\circ\phi^{-1}\in\ZZ\ml(\{(\infty),(0),(1)\})$,
the result follows.
\end{proof}

\begin{rmk}
Let $D=\set{(\infty),(0),(1)}$. The lifted Kiss--Somlai set
satisfies
\[
    \mu_{\Lift(K)}\notin\ZZ\ml(D),
    \qquad
    p\mu_{\Lift(K)}\in\ZZ\ml(D).
\]
The second assertion follows from Lemma~\ref{lm:KF-identity}.
For the first, note that every element of $\ZZ\ml(D)$ has
congruent intersection numbers modulo $q$ with all vertical
lines. However, choose $a,b\in\fq$ such that $\nu(\Tr(a))=p-1$ and $\nu(\Tr(b))=p-2$. The vertical lines $X=a$ and $X=b$ meet $\Lift(K)$
in $0$ and $q/p$ points, respectively.

Negative line coefficients may be necessary even when the
multiset has non-negative multiplicities. For example, taking
$q=p$, we have $p\mu_K\in\ZZ\ml(D)$, but this multiset cannot
be expressed as a non-negative linear combination of lines.
Indeed, $K$ is disjoint from both $X=-1$ and $Y=0$, and hence
contains no affine line. Thus any line with a positive
coefficient in such a combination would give positive
multiplicity at a point outside $K$.
\end{rmk}


For ordinary sets, we get the following stronger result.

\bigskip 

\noindent {\bf Theorem \ref{Thm:Main:Lift}.}
{\it
 Let $p$ be a prime number and $q = p^h$ a power of $p$.
 Let $\Tr:\FF_q\to\FF_p$ be the absolute trace function.
 Recall the natural total ordering $<$ on the elements of $\FF_p$.
 Then the set 
 \[
  L = \sett{(x,y) \in \FF_q^2}{\Tr(x) < \Tr(y)}
 \]
 of points in $\ag(2,q)$ is inequidistributed from exactly 3 directions: $(0)$, $(1)$, and $(\infty)$.
 Moreover, every set that is inequidistributed from exactly 3 directions in $\ag(2,q)$ is affinely equivalent to the union of $L$ and some parallel lines that are disjoint from $L$.
}

\begin{proof}
Since $L = \Lift(K)$, the first part of the theorem follows from Proposition \ref{prop:lift-directions} and Theorem \ref{Thm:KS}.
We still need to prove that if $T$ is a set of points in $\AG(2,q)$ which is inequidistributed from exactly 3 directions, then $T = T_1 \cup T_2$ with $T_1$ affinely equivalent to $L$ and $T_2$ a union of parallel lines disjoint from $T_1$.

The complement of $K$ is the union of the line $X=Y$ and the set $\sett{(x,y) \in \FF_p^2}{\nu(x) > \nu(y)}$ which is affinely equivalent to $K$.
Thus, if $q=p$, the assertion follows directly from Theorem \ref{Thm:KS}.
Hence suppose that $h>1$.

If $T$ contains no affine line, put $T_1=T$ and $T_2=\emptyset$.
Otherwise, let $\ell\subseteq T$ be a line with direction $(d)$. Since
$T$ has an equidistributed direction, $|T|=kq$ for some integer $k$. If
$(d)$ were equidistributed, then every line with direction $(d)$ would
contain $k$ points of $T$. As $\ell\subseteq T$, this would imply $k=q$
and hence $T=\fq^2$, a contradiction. Thus $(d)$ is inequidistributed.
Let $T_2$ be the union of all lines parallel to $\ell$ that are contained
in $T$, and put $T_1=T\setminus T_2$. Removing $T_2$ changes all
line-intersection numbers in every other direction by the same amount.
The direction $(d)$ also remains inequidistributed: otherwise, since
$T_1$ is disjoint from at least one line with direction $(d)$, it would
be empty, and $T$ would have only one inequidistributed direction.
Therefore, $T_1$ has the same three inequidistributed directions as $T$.

In either case, $T_1$ contains no affine line. Indeed, all complete lines
parallel to $\ell$ were removed, while every line not parallel to $\ell$
meets one of the removed lines.
Applying an affine transformation, we may assume that the three
inequidistributed directions are $(\infty)$, $(0)$, and $(1)$.
Apply the proof of Theorem \ref{thm:integer-three-directions} to $T_1$.
Since $T_1$ contains no affine line,
$f_d^*(t)\in\{0,\dots,q-1\}$ for all $d\in\{\infty,0,1\}$
and all $t\in\fq$.
By (\ref{Eq:f's and A's}), $a=0$ would force each $f_d^*$
to be constant, a contradiction; hence, after applying
$(x,y)\mapsto(ax,ay)$, we may assume that $a=1$.  
Together with (\ref{Eq:f's and A's}), this shows that $A_d(t)$ is uniquely defined by $f^*_d(0)$ and $\Tr(t)$.
Therefore, $\mu_{T_1}(x,y)$ only depends on $\Tr(x)$ and $\Tr(y)$, and hence $T_1 = \Lift(K')$ for some set $K' \subseteq \FF_p^2$.
By Proposition \ref{prop:lift-directions}, $K'$ has exactly three inequidistributed
directions.
Thus, by Theorem \ref{Thm:KS}, $K'$ is the image of $K$ or its complement under an affine transformation.
Since $T_1$ does not contain any line, $K'$ must be affinely equivalent to $K$. 
Affine transformations of $\AG(2,p)$ lift to affine transformations of
$\AG(2,q)$ (see Lemma \ref{lm:lift-equivariance}), thus $T_1$ is affinely equivalent to $\Lift(K) = L$.
Adding 
$T_2$ back proves the assertion.
\end{proof}

\subsection{Translation planes}

A translation in an affine plane $A_q=(\mp,\ml)$ is either the identity,
or a collineation that maps each line to a parallel line and has no fixed
points. Equivalently, it is an elation in the projective extension with
$\ell_\infty$ as axis. The translations of $A_q$ form a group
$T(A_q)$, and for any $P,Q\in\mp$, there can be at most one translation
mapping $P$ to $Q$. If $T(A_q)$ acts transitively, and hence regularly,
on $\mp$, then $A_q$ is called a translation plane. Bruck and Bose
\cite{Bruck:Bose} gave the following characterisation of translation planes.

\begin{const}[{\cite[Theorems~4.1 and~7.1]{Bruck:Bose}}]
\label{const:ABB}
Consider the projective space $\PG(2k,q)$, and take a hyperplane
$\Pi_\infty$. Let $\ms$ be a spread of $(k-1)$-spaces in $\Pi_\infty$,
i.e.\ a set of $q^k+1$ pairwise disjoint $(k-1)$-spaces that partition the
points of $\Pi_\infty$. Let $\mp$ be the points of $\PG(2k,q)$ not in
$\Pi_\infty$, and consider the collection of $k$-spaces of $\PG(2k,q)$
that intersect $\Pi_\infty$ in an element of $\ms$. Let $\ml$ be the
family of point sets obtained from this collection after removing their
points on $\Pi_\infty$. Then $(\mp,\ml)$ is a finite affine translation
plane of order $q^k$. Conversely, every finite affine translation plane can
be constructed in this way.
\end{const}

We call Construction \ref{const:ABB} the Andr\'e/Bruck--Bose construction, or ABB construction for short.
We first recall an elementary lemma.
It follows e.g.\ from \cite[Theorem 1.6]{Thas}, but we include a short proof.

\begin{lm}\label{lm:three-skew-subspaces}
The group $\GL(2k,q)$ acts transitively on triples of pairwise skew
$k$-spaces in $\fq^{2k}$.
\end{lm}

\begin{proof}
Suppose that $K_1,K_2,K_3$ are $k$-spaces of $\fq^{2k}$, any two of which
intersect trivially. Choose a basis $v_1,\ldots,v_k$ of $K_1$. For every
$i$, the $(k+1)$-spaces $\vspan{v_i,K_2}$ and $\vspan{v_i,K_3}$ intersect
in a $2$-space through $v_i$. Choose non-zero vectors $v_{k+i}\in K_2$ and $u_i\in K_3$
in this $2$-space. After rescaling $v_i$ and $v_{k+i}$, we may assume that
\[
    u_i=v_i+v_{k+i}.
\]
The vectors $u_1,\ldots,u_k$ are linearly independent: if
$\sum_i\gamma_i u_i=0$, then
\[
    \sum_i\gamma_i v_i=-\sum_i\gamma_i v_{k+i}.
\]
The left-hand side belongs to $K_1$ and the right-hand side to $K_2$; since
$K_1\cap K_2=\{0\}$, all $\gamma_i$ are zero. Similarly,
$v_{k+1},\ldots,v_{2k}$ are linearly independent. Consequently, there is a
basis $v_1,\ldots,v_{2k}$ such that
\[
\begin{aligned}
    K_1&=\vspan{v_1,\ldots,v_k},\\
    K_2&=\vspan{v_{k+1},\ldots,v_{2k}},\\
    K_3&=\vspan{v_1+v_{k+1},\ldots,v_k+v_{2k}}.
\end{aligned}
\]
Since $\GL(2k,q)$ acts transitively on bases, it acts transitively on such
triples.
\end{proof}

\begin{crl}\label{crl:three-directions-representation}
Let $A_q$ be a translation plane of order $q$, and let
$(d_1),(d_2),(d_3)$ be three distinct directions in $A_q$. Then $A_q$
can be represented as $(\fq^2,\ml)$ so that its lines with directions
$(d_1),(d_2),(d_3)$ are the lines of $\AG(2,q)$ with directions
$(\infty),(0),(1)$, respectively.
\end{crl}

\begin{proof}
Both $A_q$ and $\AG(2,q)$ are translation planes, so $q=p^h$ for some
prime $p$. They can both be obtained from the ABB construction in
$\PG(2h,p)$. Let $\ms_1$ and $\ms_2$ be the spreads corresponding to
$A_q$ and $\AG(2,q)$, respectively. By Lemma
\ref{lm:three-skew-subspaces}, there is a collineation of $\PG(2h,p)$ that
maps the three pairwise disjoint $(h-1)$-spaces of $\ms_1$ corresponding to
$(d_1),(d_2),(d_3)$ to the elements of $\ms_2$ corresponding to
$(\infty),(0),(1)$. Applying this collineation to $\ms_1$ still yields an
affine plane isomorphic to $A_q$, and the statement follows.
\end{proof}

Corollary \ref{crl:three-directions-representation} allows us to represent
the translation plane and $\AG(2,q)$ on the same point set so that the
three prescribed parallel classes coincide. The orthogonal decomposition
from Lemma \ref{Lm:Ortho}, applied to multiplicity functions, then shows
that an integer multiset whose inequidistributed directions are precisely
these three directions has the same inequidistributed directions in both
planes. 
We therefore obtain the following.

\begin{crl}\label{crl:translation-three-directions}
Let $A_q$ be a translation plane of order $q=p^h$, and let
$(d_1),(d_2),(d_3)$ be three distinct directions.
Choose the representation from Corollary
\ref{crl:three-directions-representation}.
Then an integer multiset on $\fq^2$ has exactly
$(d_1),(d_2),(d_3)$ as its inequidistributed directions in $A_q$
if and only if it has exactly $(\infty),(0),(1)$ as its
inequidistributed directions in $\AG(2,q)$.
Consequently, Theorems \ref{thm:integer-three-directions}
and \ref{Thm:Main:Lift} apply to these multisets and ordinary
point sets, respectively, when viewed in $\AG(2,q)$.
\end{crl}

\section*{Use of artificial intelligence}


The proof of Theorem
\ref{thm:integer-three-directions} 
and the covariance method from Appendix \ref{Sec:Covariance}, as an alternative to the linear algebraic arguments from Section \ref{Sec:characteristic}, were based on  suggestions made by GPT 5.5 through ChatGPT.

\section*{Acknowledgements}

Sam Adriaensen was supported by grant 12A3Y25N of Research Foundation Flanders (FWO).

This paper was supported by the János Bolyai Research Scholarship of the Hungarian Academy of Sciences. Project nos. 153080, 151504 and 152582 have been implemented with the support provided by the Ministry of Culture and Innovation of Hungary from the National Research, Development and Innovation Fund, financed under the ADVANCED\_25, EXCELLENCE\_24 and SNN\_25 funding schemes, respectively.

We thank Vladislav Taranchuk for his help.

\newcommand{\etalchar}[1]{$^{#1}$}

\appendix

\section{Covariance as an alternative to eigenspace decomposition}
 \label{Sec:Covariance}

Many inequalities in finite geometry -- upper or lower bounds on the size of point sets with a certain property -- can be proved using the so-called variance method, see e.g.\ \cite{Stinson} for a survey on this method.
The variance method is particularly useful when the equality case forces the point set to have only two possible intersection sizes with lines (or blocks).
Typically, the variance method is equivalent to an application of the bipartite expander mixing lemma, see e.g.\ \cite{DeWinter, Bishnoi}, which is an eigenvalue bound from algebraic graph theory.

Another common approach in finite geometry is to study sets $S$ and $T$
using an eigenspace decomposition of a matrix associated with the
geometry. If their characteristic vectors, after subtracting their
respective mean values, are orthogonal, then their intersection size
equals the product of their sizes divided by the total number of
points in the geometry; see e.g.\ \cite[Theorem 4]{Bamberg}.
In this section, we give a combinatorial alternative for proving such equalities.

A $2-(v,k,\lambda)$ \emph{design} is a triple $(\mp,\mb,I)$, where $\mp$ is a set whose elements are called \emph{points}, $\mb$ is a set whose elements are called \emph{blocks}, and $I \subseteq \mp \times \mb$ is the \emph{incidence relation}, such that $|\mp| = v$, every block is incident with $k$ points, and every two distinct points are incident with $\lambda$ blocks.
In that case, every point is incident with $r = \lambda \frac{v-1}{k-1}$ blocks; the total number of blocks is given by $b = vr/k$.
Given two integer multisets $S$ and $T$ of points in a $2-(v,k,\lambda)$ design, define their \emph{covariance} as
\[
 \Gamma(S,T) = \sum_{B \in \mb} \left( |S \cap B| - |S| \frac kv \right) \left( |T \cap B| - |T| \frac kv \right).
\]
We use the name covariance because, given the random variable $X$ uniformly distributed over $\mb$, $\Gamma(S,T)/b$ is the covariance of the random variables $|S \cap X|$ and $|T \cap X|$.

\begin{lm}
 \label{Lm:Cov}
 Given two integer multisets $S$ and $T$ of points in a $2-(v,k,\lambda)$ design $(\mp,\mb,I)$,
 \[
  \Gamma(S,T) = (r-\lambda) \sum_{P \in \mp} \mu_S(P) \mu_T(P) - \frac rv \frac{v-k}{v-1} |S| |T|.
 \]
 If $S$ and $T$ are ordinary sets, this simplifies to $\Gamma(S,T) = (r - \lambda) |S \cap T| - \frac rv \frac{v-k}{v-1} |S| |T|$.
\end{lm}

\begin{proof}
 By expanding the sum in the definition of covariance, we find
 \begin{align*}
  \Gamma(S,T) = \sum_{B \in \mb} |S \cap B| |T \cap B| - |S| \frac kv \sum_{B \in \mb} |T \cap B| - |T| \frac kv \sum_{B \in \mb} |S \cap B| + b |S| |T| \frac{k^2}{v^2}.
 \end{align*}
Note that $\sum_{B \in \mb} |S \cap B|$ counts incident pairs $(P,B) \in I$, each with weight $\mu_S(P)$. 
Counting in another way, this gives $|S|r$.
Similarly, $\sum_{B \in \mb} |T \cap B| = |T|r$.
In addition, $\frac kv = \frac rb$, hence $b \frac{k^2}{v^2} = \frac{rk}v$.
Lastly, 
\begin{align*}
 \sum_{B \in \mb} |S \cap B| |T \cap B|
 &= \sum_{B \in \mb} \left( \sum_{P \in B} \mu_S(P) \right) \left( \sum_{Q \in B} \mu_T(Q) \right) 
 = \sum_{P \in \mp} \mu_S(P) \sum_{Q \in \mp} \mu_T(Q) |\sett{B \in \mb}{P,Q \; I \; B}| \\
 &= \sum_{P \in \mp} \mu_S(P) \left( \mu_T(P) r + \sum_{Q \in \mp \setminus \{P\}} \mu_T(Q) \lambda \right)
 = \sum_{P \in \mp} \mu_S(P) \left( \mu_T(P) (r- \lambda) + |T| \lambda \right) \\
 &= |S| |T| \lambda + (r-\lambda) \sum_{P \in \mp} \mu_S(P) \mu_T(P)
\end{align*}
Combining this, we obtain
\[
 \Gamma(S,T) = (r-\lambda) \sum_{P \in \mp} \mu_S(P) \mu_T(P) + \left( \lambda - \frac{rk}v \right) |S||T|.
\]
Note that $\lambda = \frac{r(k-1)}{v-1}$, so
\[
 \lambda - \frac{rk}v = r \left( \frac{k-1}{v-1} - \frac kv \right) = -\frac rv \frac{v-k}{v-1}.
\]
This yields the claimed equality. 
\end{proof}

We will now use it to give alternative proofs of the results in Section \ref{Sec:characteristic}.

\bigskip 

\noindent {\bf Theorem \ref{Thm:No shared directions}.}
 {\it
 Let $A_q = (\mp,\ml)$ be an affine plane of order $q$.
 Let $S$ and $T$ be integer multisets of points in $A_q$.
 Suppose that there is no direction from which both $S$ and $T$ are inequidistributed.
 Then
 \[
  \sum_{P \in \mp} \mu_S(P) \mu_T(P) = \frac{|S| |T|}{q^2}.
 \]
 If $S$ and $T$ are ordinary sets, this formula simplifies to $|S \cap T| = |S| |T| / q^2$.
 In particular, if $S$ and $T$ are disjoint point sets of size $q$, then there is a direction determined by both of them.
 }

\begin{proof}
 Take a line $\ell$, and call its direction $(d)$.
 Then either $S$ is equidistributed from $(d)$ and $|S \cap \ell| = \frac{|S|}q$ or $T$ is equidistributed from $(d)$ and $|T \cap \ell| = \frac{|T|}q$.
 Therefore,
 \[
  \Gamma(S,T) = \sum_{\ell \in \ml} \left( |S \cap \ell| - \frac{|S|}q \right) \left( |T \cap \ell| - \frac{|T|}q \right) = 0.
 \]
 On the other hand, $A_q$ is a $2-(q^2,q,1)$ design with $r = q+1$.
 Lemma \ref{Lm:Cov} implies that $0 = \Gamma(S,T) = q \sum_{P \in \mp} \mu_S(P) \mu_T(P) - \frac{|S| |T|}q$, which yields the claimed equality.
\end{proof}

{\bf Proposition \ref{Prop:Shared directions}.}
{\it
 Suppose that $A_q = (\mp,\ml)$ and $A'_q = (\mp,\ml')$ are affine planes of order $q$ on the same point set.
 Suppose that $\ml$ and $\ml'$ share some parallel classes of lines.
 We may assign the same direction to each shared parallel class of lines in $A_q$ and $A_q'$; let $D$ denote the set of shared directions.
 Let $S$ be an integer multiset of points.
 Suppose that in $A_q$, the inequidistributed directions of $S$ are a subset $D' \subseteq D$.
 Then also in $A'_q$, $D'$ is the set of inequidistributed directions of $S$.
}

\begin{proof}
 It suffices to prove that every line $\ell' \in \ml' \setminus \ml$ intersects $S$ in $|S|/q$ points.
 Take a direction $(d) \in D$ and a line $\ell$ with direction $(d)$.
 Then $\ell$ and $\ell'$ are not parallel in $A_q'$, hence $|\ell \cap \ell'| = 1$.
 Thus, in $A_q$, $\ell'$ is equidistributed from all directions $(d) \in D$.
 By Theorem \ref{Thm:No shared directions}, $|S \cap \ell'| = |S| |\ell'| / q^2 = |S| / q$.
\end{proof}

\end{document}